\documentclass[12pt]{article}
\usepackage{amsmath,amssymb}
\usepackage{amsthm}

\def\MC{\mathcal{M}}
\def\XC{\mathcal{X}}
\def\FC{\mathcal{F}}
\def\E{\mathbf{E}}
\def\N{\mathbf{N}}
\def\P{\mathbf{P}}
\def\R{\mathbf{R}}
\def\Z{\mathbf{Z}}
\def\x{\mathbf{x}}
\def\y{\mathbf{y}}
\def\1{\mathbf{1}}
\def\z{\mathbf{z}}

\def\tr{\rm{tr}}
\def\al{\alpha}
\def\pa{\partial}
\def\de{\delta}
\def\ka{\varkappa}

\newcommand{\si}{\sigma}
\newcommand{\Si}{\Sigma}
\newcommand{\Om}{\Omega}
\newcommand{\La}{\Lambda}

\numberwithin{equation}{section}
\newtheorem{prop}{Proposition}
\newtheorem{theorem}{Theorem}
\newtheorem{lemma}{Lemma}

\newtheorem{remark}{Remark}

\begin{document}
\title{Fractional kinetic equations}
%\author{V. N. Kolokoltsov\thanks{St. Petersburg State University, St. Petersburg, 199034 Russia\\
%		 National Research University Higher School of Economics, Moscow, 101000 Russia, Email: v.n.kolokoltsov@gmail.com}\,
%	and M. S. Troeva\thanks{North-Eastern Federal University,  Yakutsk, 677000 Russia, Email: troeva@mail.ru}
%		}

\author{
	Vassili N. Kolokoltsov\thanks{Department of Statistics, University of Warwick,
		Coventry CV4 7AL UK, and National Research University Higher School of Economics, Moscow, v.kolokoltsov@warwick.ac.uk}
	and Marianna Troeva\thanks{Research Institute of Mathematics, North-Eastern Federal University,
		Yakutsk, Russia, troeva@mail.ru
	}}	

\maketitle

\begin{abstract}
We develop the idea of non-Markovian CTRW (continuous time random walk) approximation to the
evolution of interacting particle systems, which leads to a general class of fractional kinetic
measure-valued evolutions with variable order.
 We prove the well-posedness of the resulting new equations and present
 a probabilistic formula for their solutions. Though our method are quite general,
 for simplicity we treat in detail only the fractional versions of the
 interacting diffusions. The paper can be considered as a development of the ideas
 from the works of Belavkin and Maslov devoted to Markovian (quantum and classical) systems of interacting particles.
\end{abstract}

{\bf Key words:} fractional kinetic equations, interacting particles,
fractional derivative of variable order,
continuous time random walks (CTRW).

{\bf Mathematics Subject Classification (2010)}: {34A08, 35S15, 45G15}

%{UDK 517.955+517.986.7+519.214+519.217+536.95}

\section{Introduction}

The derivation of kinetic equations for the interacting particles as the dynamic law of large numbers
for Markovian models of interacting particles was seemingly initiated in \cite{Leon} and further developed
by numerous authors in a variety of forms, see e.g. \cite{BelKol, BelMas,Kac, MasTar}.
In this paper we aim to develop these approach for semi-Markovian systems of interacting particles modeled
as continuous time random walks (CTRW) with non-exponential waiting times. The resulting new kinetic equations
turn out to be fractional in time of variable position dependent fractional order.
All general information on fractional calculus that we are using can be found in the
books \cite{Scalabook,Kir94,Podlub99}.

As was stressed in \cite{Ko22},
the standard diffusion process governed by the equation
\begin{equation}
\label{eqdifeq}
\frac{\pa u}{\pa t}(t,x) =\frac12 a(x) \frac{\pa ^2 u}{\pa x^2}(t,x),
\end{equation}
with a positive continuous function $a(x)$,
can be obtained as the limit of appropriate
random walks. These prelimiting random walks can be quite different. For example,
approximating random walks in continuous time can be chosen as jump type processes with the generators
\[
L_hf(x)=\frac{1}{2h^2}[f(x+ h\sqrt{a(x)})+f(x- h\sqrt{a(x)})-2f(x)],
\]
or as jump type processes with the generators
\[
\tilde L_hf(x)=\frac{a(x)}{2h^2}[f(x+ h)+f(x- h)-2f(x)].
\]
Indeed, the operators $L_h$ and $\tilde L_h$ tend to $(1/2)a(x)(d^2/dx^2)$ as $h\to 0$, as well as the jump-type processes generated by operators $L_h$ and $\tilde L_h$ converge to the diffusion generated by the operator $(1/2)a(x)(d^2/dx^2)$.
In the first approximation, the diffusion coefficient $a(x)$ is responsible for the size of jumps (with constant intensity),
and in the second approximation, it is responsible for the intensity of jumps (with constant sizes).
The "rough" diffusion limit does not feel the difference.
The situation changes, if we model jump-type approximations via CTRW with non-exponential waiting times.
If we make the diffusion coefficient responsible for the size of jumps and take
waiting times from the domain of attraction of an $\alpha$-stable law with a constant intensity $\al$,
then the standard scaling would lead (in the limit
of small jumps and large intensities) to the most standard fractional diffusion equation

\begin{equation}
\label{eqdifeqfr}
D^{\al}_{0+*} u(t,x) =\frac12 a(x) \frac{\pa ^2 u}{\pa x^2}(t,x),
\end{equation}
where $D^{\al}_{0+*}$ is the so-called Caputo-Dzerbashyan fractional derivative.

If we will use the CTRW approximations with fixed jump sizes and use $a(x)$
to distinguish intensities at different points, then (as will be shown) we get in
the limit the equation with a variable position-dependent fractional derivative:
\begin{equation}
\label{eqdifeqfr1}
D^{\al a(x)}_{0+*} u(t,x) =\frac12 \frac{\pa ^2 u}{\pa x^2}(t,x).
\end{equation}

Of course, for any decomposition of the diffusion coefficient in a product $b(x)c(x)$ with positive functions $b(x),c(x)$, one can use $b(x)$ as a function controlling the intensity of jumps and $c(x)$ as a function controlling the spread of the jumps.
In this scenario we get in the limit the equation
\begin{equation}
\label{eqdifeqfr2}
D^{\al b(x)}_{0+*} u(t,x) =\frac12 c(x) \frac{\pa ^2 u}{\pa x^2}(t,x).
\end{equation}

In \cite{Ko22} the rigorous derivation of equations of type \eqref{eqdifeqfr2} was given.
In the present paper, following the same idea of the separation of diffusion coefficients in parts
standing for times and sizes of jumps, we derive the fractional kinetic equations of variable order
describing the scale limits of CTRW approximating systems of interacting diffusions.

Kinetic equations are measure-valued equations expressing the dynamic law of large number (LLN) limit of Markovian
interacting particle systems when the number of particles tends to infinity. The resulting nonlinear measure-valued
evolutions can be interpreted probabilistically as nonlinear Markov processes, see \cite{Ko10}.
In case of a finite state space, the set of probability measures coincides with a simplex $\Si_n$ of sequences
of nonnegative numbers $x=(x_1, \cdots, x_n)$ such that $\sum_j x_j=1$, where $n$ can be any natural number.
In a discrete case, the kinetic equations, describing the dynamic law of large number (LLN) limit of Markovian
interacting particle systems when the number of particles tends to infinity, are ODEs of the form
\begin{equation}
\label{eqkindis}
\dot x= xQ(x) \Longleftrightarrow \dot x_i=\sum_k x_k Q_{ki}(x) \,\, \text{for all} \,\, i,
\end{equation}
where $Q$ is a stochastic or Kolmogorov $Q$-matrix (that is, it has non-negative non-diagonal elements
and the elements on each row sum up to zero) depending Lipschitz continuously on $x$.

The evolution of functions $F(t,x)=F(X_t(x))$, where $X_t (x)$ denotes the solution of the equation (\ref{eqkindis}) with the initial condition $x$, satisfies the equation
\begin{equation}
\label{eqkindisfu}
\frac{\pa}{\pa t} F(t,x)=(x,Q(x) \nabla F(t,x))=\sum_{k,i}x_k Q_{ki}(x)\frac{\pa F(t,x)}{\pa x_i}.
\end{equation}

More generally (see \cite{Ko10}), for a system of mean-field interacting particles given by a family of operators $A_{\mu}$,
which are generators of Markov processes in some Euclidean space $\R^d$ depending on probability measures $\mu$ on $\R^d$
as parameters and having a common core, the natural scaling limit of such system, as the number of particles tends to infinity
(dynamic LLN), is
described by the kinetic equations in the weak form
\begin{equation}
\label{eqkin}
(f, \dot \mu_t) = (A_{\mu_t}f,\mu_t),
\end{equation}
where $f$ is an arbitrary function from the common core of the operators $A_{\mu}$.

The corresponding generalization of equation \eqref{eqkindisfu} is the differential equation
in variational derivatives:
\begin{equation}
\label{eqkinfu}
\frac{\pa F(t,\mu)}{\pa t} =\int_{\R^d} \left(A_{\mu} \frac{\de F(t,\mu)}{\de \mu (.)}\right)(z) \mu(dz).
\end{equation}

One can distinguish the cases of evolutions preserving the number of particles
and changing the number of particles. The second case arises from pure jump processes.
It was treated in \cite{KoTr21}. In this paper we shall deal with the first case.

\begin{remark} The proof of the convergence result in \cite{KoTr21} has a gap. Here we use a slightly
different method (which can be used also to correct the arguments of \cite{KoTr21})
 based on the ideas from \cite{Ko09,Ko22}.
\end{remark}

All these equations are derived as the natural scaling limits of some random walks on the space of the
configurations of many-particle systems.

  When the standard random walks are extended to more general CTRWs
(continuous time random walks), characterised by the property that the random times between
jumps are not exponential, but with tail probabilities decreasing as a power function,
their limits turn to non Markovian processes described by the fractional evolutions.

For instance, one can write the Caputo-Dzherbashyan fractional derivative of some fixed order $\al$ instead
of the usual one in \eqref{eqkin}, as was done e.g. in \cite{KochKondr17}.
However, taking into account the natural possibility of different waiting times for jumps from different states
(as for usual Markovian approximation), one would obtain an equation
of a more complicated type, with the fractional derivatives of position-dependent order.
In \cite{KolMal} this derivation was performed for the case of a discrete state space, that is for a nonlinear Markov chain
described by equation \eqref{eqkindis}, leading to the fractional generalization of equation \eqref{eqkindisfu} of the form
\begin{equation}
\label{frackineqdis}
D_{t-\star}^{(x,\al)} F(s,x)=(x,Q(x) \nabla F(s,x)),
\end{equation}
where $\al=(\al_1, \cdots, \al_n)$ is the vector describing the power laws of the waiting times in
various states $\{1, \cdots, n\}$ and $D_{t-\star}^{(x,\al)}$ is the right Caputo-Dzherbashyan fractional
derivative of order $(x,\al)$ depending on the position $x$ and acting on the time variable $t$.

As we shall show, the corresponding limiting equation for the mean field interacting
particle systems in $\R^d$ preserving the number of particles writes down as the equation
\begin{equation}
\label{eqkinfufrac}
D_{t-\star}^{(\al,\mu)} F(s,\mu) =\int_{\R^d} \left(A_{\mu} \frac{\de F(s,\mu)}{\de \mu (.)}\right)(z) \mu(dz),
\end{equation}
where $A_{\mu}$ is the family of generators of Feller processes in $\R^d$ depending on $\mu\in \MC(\R^d)$.
We also supply the well-posedness of this equation and the probabilistic formula for the solutions.
For transparency we shall deal only with the case of diffusion operators $A_{\mu}$.

The content of the paper is as follows. In the next section we recall preliminary material from the theory of kinetic equations. In Section \ref{CTRWmodeling} we construct the CTRW approximations to the systems of mean-field interacting particles for the case of non-exponential waiting times between jumps. In Section \ref{mainresults} we formulate our main results on the convergens of  CTRW approximations to the new kinetic equations with fractional variable order. We also supply a probabilistic formula for the solutions of these equations. In Sections \ref{prooflemma}-\ref{proofTh2} we give the proves of the main results. In additional Section \ref{append} we recall (in a convenient for us form) the jump-type approximations to stable generators.

The bold letters $\E$ and $\P$ will be used to denote expectation and probability.
If not mentioned otherwise, the norm on functions $\|f\|$ denote the usual sup-norm.

For a locally compact space $X$, we denote by $C(X)$ the Banach space of bounded continuous functions
on $X$, by $C_{\infty}(X)$ its subspace of functions vanishing at infinity,
and by $\MC(X)$ the space of bounded Borel measures on $X$. For $f\in C(X)$ and $\mu\in \MC(X)$,
 the usual pairing is given by the integration:
 \[
 (f,\mu)=\int f(x) \mu (dx).
\]

For a locally compact space $X$ let us denote by $C^k(\MC(X))$  the space of functions
$F$ on $\MC(X)$ such that the variational derivatives
$\de^lF(Y)/ \de Y(x_1)\cdots \de Y(x_l)$ are well defined for all $l\le k$, $Y\in \MC(X)$,
and represent a continuous mapping of $(k+1)$ variables
(measures are equipped with their weak topology) and
$\de^lF(Y)/ \de Y(.)\cdots \de Y(.)$ belong to $C_{\infty}(X^l)$ uniformly for $Y$ from bounded sets.

We will work mostly with the case $X=\R^d$. Then we will denote by  $C^k_{\infty}(X)$ the subspace of $C_{\infty}(X)$ consisting
the functions that have derivatives up to order $k$ that belong to the $C_{\infty}(X)$. We shall denote by $C^{1;k}(\MC(X))$
the subspace of $C^1(\MC(X))$ consisting of the functions  $F$
such that
the variational derivatives $\de F(Y)/ \de Y(.)$ belong to the space $C^k_{\infty}(X)$ uniformly for $Y$ from bounded sets.

By $C^{2;1\times 1}(\MC(X))$ we denote the subspace of $C^2(\MC(X))$ consisting of the functions $F$
such that
the derivatives
\[
\frac{\pa}{\pa x}\frac{\de^2F(Y)}{\de Y(x) \de Y(y)},
\frac{\pa}{\pa y}\frac{\de^2F(Y)}{\de Y(x) \de Y(y)},
\frac{\pa^2}{\pa x \pa y}\frac{\de^2F(Y)}{\de Y(x) \de Y(y)}
\]
are well defined and belong to $C_{\infty}(X^2)$.

\section{Preliminaries: kinetic equations}

Let us recall some basic notations and facts from the standard measure-valued-limits
 of interacting processes.

For simplicity let us take the state space of one particle as the Euclidean space $X=\R^d$.
Denoting by $X^0$ a one-point space and by $X^j$ the powers $X\times
\cdots \times X$ ($j$ times),
we denote by ${\cal X}$ their disjoint union
$\XC=\cup_{j=0}^{\infty} X^j$, which stands for the state
space of a random number of similar particles.
We denote the elements of $\XC$ by bold letters, say $\x$, $\y$.

Let us denote by  $S\XC$ (or $SX^k$ resp.) the quotient
space of $\XC$ (or $X^k$ resp.) obtained by factorization with
respect to all permutations.

A key role in the theory of measure-valued limits of interacting
particle systems is played by the scaled inclusion $S\XC$ to $\MC(X)$ given by
\begin{equation}
\label{eqcorrsetparticlesandmes}
 \x=(x_1,...,x_N) \mapsto
(\delta_{x_1}+\cdots +\delta_{x_N})/N=\delta_{\x}/N,
 \end{equation}
 which defines a bijection between $S\XC$ and the set
 $\MC^+_{\delta}(X)$ of normalized finite sums of Dirac's
 $\de$-measures.

\begin{remark}
Let us stress that we are using here a non-conventional notation:
 $\de_{\x}=\de_{x_1}+\cdots +\de_{x_N}$, which is convenient for our purposes.
\end{remark}

Let $L^{\rm dif}_{\mu}$ be a family of diffusion operators in $C_{\infty}(X)$ ($X=\R^d$), i.e.
\begin{equation}
\label{diffone}
L_{\mu}^{\rm dif}f(x) =\frac12 {\tr} \left(G(x) \frac{\pa^2 f}{\pa x^2}\right)
+\left(b_{\mu}(x), \frac{\pa f}{\pa x}\right),
\end{equation}
with the drift coefficient depending continuously
on $\mu\in \MC(X)$.

 Let us choose some jump-type approximations $L^h_{\mu}$ to these operators, that is,
 \begin{equation}
\label{diffoneapp0}
L_{\mu}^hf(x)=\frac{1}{h^2}\int [f(x+hy)-f(x)] p_{\mu}(x,dy),
\end{equation}
with some probability kernels $p_{\mu}(x,dy)$,
  so that
\begin{equation}
\label{diffoneapp}
\|L_{\mu}^{\rm dif}f-L_{\mu}^{h}f\| \to 0, \quad \text{as} \quad h\to 0,
\end{equation}
for all $f\in D$.

The simplest such approximation is of the form
\[
L_{\mu}^hf(x)=\frac{1}{h^2}\sum_{i=1}^d [f(x+hb_ie_i)+f(x-hb_ie_i)-2f(x)]
\]
\begin{equation}
\label{diffoneapp1}
+\frac{1}{h^2}\sum_{i>j}[f(x+hb_{ij}(e_i+e_j)+f(x-hb_{ij}(e_i+e_j)-2f(x)],
 \end{equation}
with appropriate functions $b_i, b_{ij}$ (depending on $x$ and $\mu$),
where $\{e_i\}$ is the standard basis in $\R^d$.

As our basic one-particle generators we choose operators of the form
\begin{equation}
\label{diffonefi}
A_{\mu}f(x)=a(x) L_{\mu}^{\rm dif}f(x),
\end{equation}
with some  bounded from above and below positive function $a(x)$. Let us stress again that the separation
of the multiplier $a(x)$ is not intrinsic. It just shows our choice of the part of the diffusion generator that
will be responsible for the waiting times of jumps in our future CTRW approximation. The corresponding approximation
will be denoted $A^h_{\mu}=a(x) L^h_{\mu}$.

 We shall work under the following conditions:

{\it Condition A.} $G(.)$ and $b_{\mu}(.)\in C^2(\R^d)$, the latter with the norms bounded for bounded $\mu$,
and $G$ is uniformly elliptic. It implies, in particular that
$D=C^2_{\infty}(\R^d)$ is a common core for all Feller semigroups, generated by $L^{\rm dif}_{\mu}$.

{\it Condition B.} The function $b_{.}(x)$ belongs to $C^{2;1\times 1}(\MC(\R^d))$ uniformly in $x$ and \linebreak
$\de b_{\mu}(.)/\de \mu(y)\in C^1(\R^d)$ uniformly in $y$.

{\it Condition C.} The function $a(x)$ belongs to $C^2(\R^d)$
and is bounded from above and below by positive constants.

\begin{remark} These conditions can be relaxed in many directions.
One can also include the case with $G$ and $a$ depending on $\mu$. Essentially what one needs is the well-posedness
of kinetic equation \eqref{eqkin} with $A=L^{\rm dif}_{\mu}$ and twice continuous differentiability of the solutions
with respect to initial data. Chapter 6 of \cite{Kobook19} contains various criteria for this to hold.
\end{remark}

To any family of one-particle differential operators $B_{\mu}$ on $C_{\infty}(\R^d)$, with coefficients depending on the mean-field $\mu$, there corresponds an operator $\hat B$ acting on the space of functions $C^{sym}(\R^{dN})$ on $N$-particle states (for any $N$) by the formula
\begin{equation}
\label{secquant}
\hat B f(\x) =\sum_{j=1}^N B^j_{\de_{\x}/N}f(\x),
\end{equation}
where $B^j_{\mu}$ denotes the operator $B_{\mu}$ acting on the variable $x_j$.

\begin{remark}
In physics this extension of $B_{\mu}$ from a one-particle states to the operator $\hat B$
on multi-particle states is referred to as the second quantization.
\end{remark}

By the transformation  \eqref{eqcorrsetparticlesandmes}, we can transfer the
operator $\hat B$ on $C(S\XC)$ to the operator $\tilde B$
on $C(\MC^+_{\de}(X))$. If $B=A$ is a diffusion operator above, then, by Proposition 9.2 of \cite{Ko10},
\begin{equation}
\label{mainapp}
\tilde A F(Y) =\int_{\R^d} \left(A_Y \frac{\de F(Y)}{\de Y(.)}\right)(z)Y(dz)+O(1/N),
\end{equation}
for $\x=(x_1, \cdots, x_N)$, $Y=\de_{\x}/N$, and any $F\in C^{2;1\times 1}(\MC(\R^d))$,
with $O(1/N)$ being uniform for $F$ from bounded subsets of $F\in C^{2;1\times 1}(\MC(\R^d))$.

This makes it plausible to conclude (proof can be found in \cite{Ko10}) that, as $N\to \infty$,
the process generated by $\tilde A$ converges weakly to a deterministic process on measures
generated by the operator on the r.h.s. of \eqref{eqkinfu}, so that this process is
given by the solution of a kinetic equation \eqref{eqkin}.
Notice that \eqref{eqkin} is obtained from \eqref{eqkinfu} by choosing
$F$ to be a linear function $F(\mu)=(f,\mu)$. We aim at the extension of this result
for CTRW approximations with non-exponential waiting times.

\begin{remark}
Evolution \eqref{eqkinfu} can be used to study the behavior of various functionals of the limit process. However, for our purpose, equation \eqref{eqkinfu} is used as an auxiliary tool that allows us to exploit the theory of operator semigroups.
\end{remark}

\section{CTRW modeling of interacting particle systems}
\label{CTRWmodeling}

Our objective is to obtain the dynamic LLN
for interacting multi-particle systems for the case of non-exponential
waiting times between jumps having a power tail distributions.

Recall that a positive random variable $\si$ with a probability law $\P$ on $[0,\infty)$
is said to have a {\it power tail of index} $\al$ if
\[
\P(\si > t) \sim  \frac{\ka}{t^{\al}}
\]
for large $t$, that is the ratio of the l.h.s. and the r.h.s tends to $1$, as $t \to \infty$.
Here $\ka>0$ is a positive constant.

The power tails are invariant under taking minima,
in the same way, as the exponential tails are. Namely, if
$\si_j$, $j=1, \cdots ,d$, are independent variables with a power tail of indices $\al_i$
and normalising constants $\ka_i$,
then $\si=\min(\si_1, \cdots, \si_d)$ is clearly a variable with a power tail of index
$\al=\al_1+ \cdots + \al_d$ and normalising constant $\ka_1\cdots \ka_d$.

In full analogy with the case of exponential times,
 let us assume here that the waiting time of the particle at $x_i$
has the power tail with the index $\al(x_i)=\al \tau a(x_i)$ with some fixed $\al \in (0,1)$.
For simplicity assume that the normalising constant $\ka$ equals $1$. Consequently,
the minimal waiting time of all $N$ points in a collection $\x=(x_1, \cdots, x_N)$
will have some probability distribution $Q_{\x}(dr)$ with a tail of the index $\al \tau A(\x)$,
where $A(\x)=\sum_i a(x_i)$.

To simplify presentation we shall make the distribution $Q_{\x}$ more precise. Namely, we assume
that these distributions have continuous densities $Q_{\x}(r)$ such that
\begin{equation}
\label{eqpowertail1}
Q_{\x}(r)=\al \tau A(\x) r^{-1-\al \tau A(\x)} \,\, \text{for}
\,\, r\ge B, \,\, \text{and} \,\, Q_{s,x}(r) \le 1 \,\, \text{for all} \,\, r,
\end{equation}
with some $B>0$ uniformly for all $\x, \tau$.

Our approximating process with power tail waiting times controlled by the function $a(x)$
and the jumps governed by the operator $L_{\mu}^h$ can be described probabilistically
as follows. Starting from any time and current state $\x$, we wait a
random waiting time $\si$, which has the distribution $Q_{\x}(dr)$.
When $\si$ rings, a particle at $x_i$ that makes a transition, is chosen
according to the probability law $a(x_i)/A(\x)$,
and then it makes an instantaneous transition to $y$
according to the distribution $p_{\de_{\x}/N}(x_i,dy)$.
 Then this procedure repeats starting from the new state
$\x\setminus x_i \cup y$.

In order to derive the LLN in this case, the usual trick is to lift
this non-Markovian evolution on the subsets $\x=(x_1, \cdots, x_N)$
or the corresponding measures $\de_{\x}/N$ to the
discrete time Markov chain on $\MC^+_{\delta}(X)\times R$
by considering the total waiting time $s$ as an additional space variable. Additionally
one has to adhere to the usual scaling of the waiting time for
the jumps of CTRW (see e.g. \cite{MS2} or \cite{Ko11}).
Thus we shall consider the Markov chain $(M^{\tau}_{\mu,s},S^{\tau}_{\mu,s})(k\tau)$
on $\MC^+_{\delta}(X)\times R_+$ with the jumps occurring at discrete times $k\tau$, $k\in \N$,
such that the process at a state $(\x,s)$ at time $\tau k$
jumps to $(\x \setminus x_i \cup (x_i+hy), s+\tau^{1/\al \tau A(\x)}r)$,
or equivalently a state $(\de_{\x}/N,s)$ jumps to the state
$(\de_{\x}-\de_{x_i}+\de_{x_i+hy})/N, s+\tau^{1/\al \tau A(\x)}r)$,
where first $x_i$ is chosen according to the law $a(x_i)/A(\x)$
and then $y$ according to the law $p_{\de_{\x}/N}(x_i, dy)$, and
where $r$ is distributed by $Q_{\x}(dr)$.

We have three natural small parameters in the model: $1/N$ related to the number of particles,
$h$ related to the size of jumps, and $\tau$ related to the waiting times. In order to obtain
a reasonable limit, we link them as $\tau=1/N=h^2$. Then
\[
\tau A(\x)=(a,\de_{\x}/N)=\int a(y) (\de_{\x}/N) (dy).
\]
Assuming that $\al \sup_x a(x) <1$, we ensure that $\al \tau A(\x)<1$ for all $\x$.

The  transition operator of the chain $(M^{\tau}_{\mu,s},S^{\tau}_{\mu,s})(k\tau)$ is given by
\begin{equation}
\label{transprobmfMCtot2}
U^{\tau}F(\mu,s)=\int_{\R^+} \int_{\R^d} Q_{\x}(dr)\sum_i p_{\mu}(x_i, dy)\frac{a(x_i)}{A(\x)}
F\left(\mu-\tau\de_{x_i}+\tau\de_{x_i+hy}, s+\tau^{1/\al \tau A(\x)}r\right),
\end{equation}
for $\mu=\tau\de_{\x}=\tau\sum_j \de_{x_j}$.

We can define the {\it scaled CTRW of mean-field interacting particles
with power-tail waiting times controlled by the function $a(x)$
and the jumps governed by the operator} $L_{\mu}^h$ as
the value of the first coordinate $M^{\tau}_{\mu,s}$  of the chain
$(M^{\tau}_{\mu,s},S^{\tau}_{\mu,s})(k\tau)$
evaluated at the random time $k\tau$ when the second coordinate
(expressing the total waiting time) $S^{\tau}_{\mu,s}(k\tau)$
reaches $t$, that is, at the time
\[
k\tau =T_{\mu,s}^{\tau}(t)=\inf \{ m\tau: S^{\tau}_{\mu,s} (m\tau) \ge t\},
\]
so that $T_{\mu,s}^{\tau}$ is the inverse process to $S^{\tau}_{\mu,s}$.

Thus this scaled CTRW is the (non-Markovian) process
 \begin{equation}
\label{scaledmfchainwithtail}
\tilde M^{\tau}_{\mu,s}(t)= M^{\tau}_{\mu,s}(T_{\mu,s}^{\tau}(t)).
\end{equation}

Our main results concern the limit of this process as $\tau\to 0$.

\section{Main results}
\label{mainresults}

Recall that we shall always assume that Conditions A-C hold true.

Let us see first of all what happens with the process $(M^{\tau}_{\mu,s},S^{\tau}_{\mu,s})(k\tau)$
in the limit $\tau =h^2\to 0$.  Namely, we are interested
in the weak limit of the chains with transitions
 $[U^{\tau}]^{[t/\tau]}$, where $[t/\tau]$ denotes the integer
 part of the number $t/\tau$, as $\tau\to 0$.

It is well known \cite{Ko11,Kal}  that if such chain converges
to a Feller process, then the generator of this limiting process
can be obtained as the limit
 \begin{equation}
\label{scalingMchainsrep}
\La F= \lim_{\tau\to 0} \frac{1}{\tau}(U^{\tau}F-F).
\end{equation}

Let us introduce the space $\hat C$ of continuous functions $F(\mu,s)$ on $\MC(\R^d)\times \R$ such that
 $F(\mu,s)$ belongs to $C^1_{\infty}(\R)$ as a function of $s$ and to $C^{2;1\times 1}(\MC(\R^d))\cap C^{1;2}(\MC(\R^d))$
 as a function of $\mu$.

\begin{lemma}
\label{lem1}
Assume that $F\in \hat C$. If $\mu=\tau \de_{\x}$ converges
weakly to some measure, which we shall also denote by $\mu$ with some abuse of notation, as $\tau\to 0$, then
 \[
\La F(\mu,s)= \lim_{\tau\to 0} \frac{1}{\tau}(U^{\tau}F-F)(\mu,s)
\]
exists and
 \begin{equation}
\label{eq1lem1}
\La F(\mu,s)=\al (a,\mu) \int_0^{\infty} \frac {F(\mu, s+r)-F(\mu,s)}{r^{1+\al (a, \mu)}} dr
+\frac{1}{(a,\mu)}\int_{\R^d} a(z)\left(L_{\mu}^{\rm dif} \frac{\de F}{\de \mu(.)}\right)(z)\mu(dz).
\end{equation}
\end{lemma}

\begin{theorem}
\label{th1}
(i) The operator \eqref{eq1lem1} generates a contraction semigroup in
the space of continuous functions $F(\mu,x)$ such that $F(\mu,.)\in C_{\infty}(\R^d)$ uniformly
on bounded subsets of $\MC(\R^d)$, with an invariant core $\hat C$.
It specifies a Markov process
$(M_{\mu,s}(t),S_{\mu,s}(t))$  in $\MC(X)\times \R$ such that

the first coordinate does not depend on $s$ (and thus can be denoted shortly $M_{\mu}(t)$),
is deterministic, and solves the kinetic equation
\begin{equation}
\label{eqkin1}
\frac{d}{dt}(f, \mu_t) = \frac{a(.)}{(a,\mu)}(L^{\rm dif}_{\mu_t}f,\mu_t);
\end{equation}

and the second coordinate is a time nonhomogeneous stable-like subordinator
generated by the time dependent family of generators
 \begin{equation}
\label{eq1th1}
\La^t_{st} g(s)=\al (a,\mu_s) \int_0^{\infty} \frac {g(s+r)-g(s)}{r^{1+\al (a, \mu_s)}} \, dr.
\end{equation}

This subordinator has a smooth transition probability density  $G(u;\mu,s; S)$
(for the transition from $s$ to $S$ during the time interval $[0,u]$).

(ii) The discrete time Markov chains $(M^{\tau}_{\mu,s},S^{\tau}_{\mu,s})(k\tau)$
given by \eqref{transprobmfMCtot2} converge weakly to the Markov process $(M_{\mu,s}(t),S_{\mu,s}(t))$,
that is the corresponding finite-dimensional distributions converge. Moreover, the convergence is uniform on
compact subsets of time, so that, in particular, for any continuous bounded function $F$,
\begin{equation}
\label{eq2th1}
\lim_{\tau\to 0, \, k\tau\to t}\E F\bigl((M^{\tau}_{\mu,s},S^{\tau}_{\mu,s})(k\tau)\bigr)
=\E F\bigl((M_{\mu,s},S_{\mu,s})(t)\bigr)
\end{equation}
uniformly for $t$ from any compact set.
\end{theorem}

\begin{remark}
An important recent contribution \cite{GoldRock} allows one to identify a proper topology
on the space of probability measures that allows one to identify appropriate spaces of
functions $F(\mu,x)$, where the Markov semigroups of the type given above become strongly continuous.
\end{remark}

Finally we can formulate our main result.

 \begin{theorem}
\label{th2}
(i) The marginal distributions of the process
 \begin{equation}
\label{eq1th2}
\tilde M_{\mu,s}(t)= M_{\mu}(T_{\mu,s}(t)),
\end{equation}
where
\[
T_{\mu,s}(t)=\inf \{ u: S_{\mu,s} (u) \ge t\}
\]
 is the inverse process to $S_{\mu,s}$,
can be expressed explicitly as
\begin{equation}
 \label{eq11th2}
\E [F(\tilde M_{\mu,s}(t))]
 =\int_{0}^{\infty} du \int_0^t dS  (t-S)^{-\al (a, M_{\mu}(u))} G(u;\mu,s; S) F(M_{\mu}(u)).
\end{equation}
Its cut-off version writes down as
\begin{equation}
 \label{eq12th2}
\E [F(\tilde M_{\mu,s}(t))\1(T_t\in [1/K,K])]
 =\int_{1/K}^{K} du \int_0^t dS  (t-S)^{-\al (a, M_{\mu}(u))} G(u;\mu,s; S) F(M_{\mu}(u)).
\end{equation}

(ii) The marginal distributions of the scaled CTRW of mean-field interacting
particles \eqref{scaledmfchainwithtail} converge
 to the marginal distributions of the process \eqref{eq1th2},
that is,
 \begin{equation}
\label{eq3th2}
\lim_{\tau \to 0}  \E F(\tilde M^{\tau}_{\mu,s}(t))= \E F(\tilde M_{\mu,s}(t)).
\end{equation}
for a bounded continuous function $F(\mu)$.

(iii) For any smooth function $F(\mu)$ (with continuous bounded variational derivative), the evolution
of averages $F(\mu,s)=\E F(\tilde M_{\mu,s}(t))$ satisfies the mixed fractional differential equation
 \begin{equation}
\label{eq4th2}
D_{t-*}^{\al (a,\mu)}F(\mu,s)
= \frac{1}{(a,\mu)}\int_{\R^d} a(z)\left(L_{\mu}^{\rm dif} \frac{\de F}{\de \mu(.)}\right)(z)\mu(dz),
\quad s \in [0,t],
\end{equation}
with the terminal condition $F(\mu,t)=F(\mu)$,
where the right fractional derivative acting on the variable $s\le t$ of $F(\mu,s)$ is defined as
 \begin{equation}
\label{eq5th2}
D_{t-*}^{\al(a,\mu)}g(s) = -\al (a,\mu) \int_0^{t-s} \frac {g(s+y)-g(s)}{y^{1+\al (a,\mu)}} dy
-\al (a,\mu) (g(t)-g(s)) \int_{t-s}^{\infty}\frac {dy}{y^{1+\al (a,\mu)}}.
\end{equation}
\end{theorem}

\section{Proof of Lemma \ref{lem1}}
\label{prooflemma}

We have
\[
\frac{1}{\tau}(U^{\tau}F-F)(\mu,s)
\]
\[
=\frac{1}{\tau}\int_{\R_+} \int_{\R^d} Q_{\x}(dr)\sum_i p(x_i, dy)\frac{a(x_i)}{A(\x)}
\left[F\left(\mu-\tau\de_{x_i}+\tau\de_{x_i+\sqrt{\tau}y}, s+\tau^{1/\al (a,\mu)}r\right)-F(\mu,s)\right]
\]
\[
=\frac{1}{\tau}\int_X  Q_{\x}(dr)
\left[F\left(\mu,s+\tau^{1/\al (a,\mu)}r\right)-F(\mu,s)\right]
\]
 \begin{equation}
\label{scalingMarlift}
+\frac{1}{\tau}\int_{\R^d} \sum_i p(x_i, dy)\frac{a(x_i)}{A(\x)}
\left[F\left(\mu-\tau\de_{x_i}+\tau\de_{x_i+\sqrt{\tau}y},s\right)-F(\mu,s)\right]+R,
\end{equation}
where the error term is
\[
R=\frac{1}{\tau}\int_{\R_+} \int_{\R^d} Q_{\x}(dr)\sum_i  p(x_i, dy)\frac{a(x_i)}{A(\x)}
\left[g_{i,y}(\mu,s+\tau^{1/\al \tau A(\x)}r)-g_{i,y}(\mu,s)\right],
\]
with
\[
g_{i,y}(\mu,s)=F\left(\mu-\tau\de_{x_i}+\tau\de_{x_i+\sqrt{\tau}y},s\right)-F(\mu,s).
\]
Assuming that $\tau \de_{\x}$ converges to some measure, which we also denote by $\mu$, as $\tau\to 0$,
we can conclude by \eqref{eq10apprrateCTRW} that the first term in   \eqref{scalingMarlift} converges, as $\tau\to 0$, to
\[
\al (a,\mu) \int_0^{\infty} \frac {F(\mu, s+r)-F(\mu,s)}{r^{1+\al (a, \mu)}} dr,
\]
whenever $F$ is continuously differentiable in $s$.
By \eqref{diffoneapp} and the definition of the variational derivative, the second term in
\eqref{scalingMarlift} converges, as $\tau\to 0$, to
\[
\sum_i \frac{a(x_i)}{A(\x)}\left(L^{\rm dif}_{\mu}\frac{\de F}{\de \mu(x_i)}\right)
=\frac{1}{(a,\mu)}\int_{\R^d}a(z)\left(L^{\rm dif}_{\mu}\frac{\de F}{\de \mu(.)}\right)(z) \mu (dz).
\]

To estimate the term $R$ we note that if $F$ has a bounded derivative in $t$
and a bounded variational derivative in $\mu$, then $g(x)$ is uniformly bounded by $\tau$
and the derivative $|\pa g/\pa s|$ is uniformly bounded. Hence by \eqref{eq10apprrateCTRW} it follows that
$R \to 0$, as $\tau \to 0$ implying \eqref{eq1lem1}.

\section{Proof of Theorem \ref{th1}}

(i) By Theorem 6.8.4 of \cite{Kobook19}, under conditions A-C equation \eqref{eqkin} with $A_{\mu}=a(.)L^{\rm dif}_{\mu}$
is well-posed in the set of probability measures, and moreover, the solutions are twice continuously differentiable
with respect to the initial data. This implies that the resolving operators for the corresponding equation \eqref{eqkinfu}
preserve the space $C^{2;1\times 1}(\MC(\R^d))\cap C^{1;2}(\MC(\R^d))$. As can be see from the proof of this result it still hold true if one
change $L_{\mu}^{\rm dif}$ by a common multiplier $1/(a,\mu)$, and thus the same holds for equation \eqref{eqkin1}. Therefore
the second term of operator \eqref{eq1lem1} generates  a deterministic process $M_{\mu}(t)$ solving \eqref{eqkin1}.
Once $\mu_t$ are well defines, the operator \eqref{eq1th1} has the standard form of a stable-like subordinator
(with a time dependent index of stability), implying the rest of statement (i). The last statement
is a well known result for stable-like processes, see e.g. \cite{Ko11}.

(ii) The convergence follows directly from Lemma \ref{lem1}, statement (i) and the standard result (see e.g. \cite{Kal})
stating that convergence \eqref{scalingMchainsrep} for $F$ from a core of the limiting process implies
the weak convergence of the corresponding discrete Markov chains to this limiting process, which is uniform
on compact subsets of time. More precisely, the result from  \cite{Kal} proves this fact for Markov chains
in locally compact spaces, but the proof is seen to be valid for Markov chains with values in the space
of probability measures on $\R^d$.

 \section{Proof of Theorem \ref{th2}}
\label{proofTh2}

(i) If $M_{\mu}(s)$ is a continuous dynamic system in some metric space
(starting in $\mu$ at time zero) and $\Si$ a positive r.v. with density $g(\si)$,
then for any $K>0$ and a continuous bounded function $F(\mu)$,
\begin{equation}
 \label{eq1lemsubs}
\E [F(M_{\mu}(\Si))\1(\Si\in [1/K,K])]= \int_{1/K}^K  F(M_{\mu}(s)) g(s) \, ds.
\end{equation}

If $V_v(s)$ is an increasing Markov process
(starting in $v$ at time zero) with
time dependent generators $A_{M_{\mu}(s)}$ depending on $M_{\mu}(s)$ and with the
transition density $G(s;\mu,v;V)$ (from $v$ at time zero to $V$ at time $s$),
then the random vector $(M_{\mu}(s),V_v(u))$ has the distribution
\begin{equation}
 \label{eq1lemsubs1}
\phi_{\mu,v}(s,u;M,V) \, dM dV=\de (M_{\mu}(s)-M) G(u;\mu,v; V) \, dV.
\end{equation}
Consequently,
\begin{equation}
 \label{eq2lemsubs1}
\frac{\pa}{\pa u} \phi_{\mu,v}(s,u;M,V)=\de (M_{\mu}(s)-M) A^*_{M_{\mu}(u)}G(u;\mu,v; V) \, dV,
\end{equation}
where $A^*$ acts on the variable $V$.

Let the process $V_v(t)$ be strictly increasing a.s., so that its generalized inverse
(or hitting times) process
\[
Z_t=\sup\{ u\ge 0: V_v(u)\le t\}=\inf\{ u\ge 0: V_v(u)> t\}, \quad t>v,
\]
is continuous. Since $(Z_t\le z)=(V_v(z)\ge t)=(V_v(z)>t)$ a.s.,
for the density $g_{\mu,v}(s,t;M,z)$ of the pair $(M_{\mu}(s),Z_t)$, $s> 0$, $t> v$,
the following formula holds:
\[
g_{\mu,v}(s,t;M,z)=  \frac{\pa}{\pa z} \int_t^{\infty} \phi_{\mu,v}(s,z;M,V) \, dV
\]
\begin{equation}
 \label{eq2lemsubs2}
=\de (M_{\mu}(s)-M) \int_t^{\infty} A^*_{M}G(z;\mu,v; V) \, dV.
\end{equation}

From \eqref{eq2lemsubs2} and  \eqref{eq1lemsubs} it follows that
\begin{equation}
 \label{eq1thsubs}
\E [F(M_{\mu}(Z_t)\1(\Z_t\in [1/K,K])]
=  \int_{1/K}^K du \int_t^{\infty} dV \, A^*_{M_{\mu}(u)}G(u;\mu,v;V) F(M_{\mu}(u)),
\end{equation}
and therefore
\[
\E [F(M_{\mu}(Z_t))\1(\Z_t\in [1/K,K])]
= \int_{1/K}^K ds \int_0^{\infty} dV \, (A_{M_{\mu}(s)}\theta_{\ge t})(V)G(s;\mu,v; V) F(M_{\mu}(s)).
\]
where $\theta_{\ge t}$ is the indicator function of the interval $[t,\infty)$.

If
\begin{equation}
 \label{eq1lemsubs2}
 A_{\mu} \phi(V) =\al (a,\mu) \int_0^{\infty} \frac{\phi (V+w)-\phi(V)}{w^{1+\al (a,\mu)}} \, dw,
\end{equation}
then
\[
A_{\mu} \theta_{\ge t}(V)=\al (a,\mu) \int_{t-V}^{\infty} \frac{dw}{w^{1+\al (a,\mu)}}
= (t-V)^{-\al (a, \mu)},
\]
for $t>V$, and $A_{\mu}\theta_{\ge t}(V)=0$ for $t\le V$.
Therefore

\[
\E [F(M_{\mu}(Z_t))\1(\Z_t\in [1/K,K])]
\]
\begin{equation}
	\label{eq2thsubs}
 =\int_{1/K}^K du \int_0^t dV (t-V)^{-\al (a, M_{\mu}(u))} G(u;\mu,v; V) F(M_{\mu}(u)),
\end{equation}
which is \eqref{eq12th2} in slightly different notations.
Passing to the limit as $K\to \infty$ in \eqref{eq12th2} yields \eqref{eq11th2}.

(ii) {\it  Step 1}. As a starting point,
we make some preliminary calculations for the subordinated Markov chains in discrete times.

Let $M(k\tau)$ be an adapted process (with values in some metric space)
on a stochastic basis $(\Om, \FC, \FC_t, P)$ and $\Si$
a stopping time with values in $\{k\tau\}$, $k\in \N$.
Let the random variables  $M(m\tau) \1(\Si=k\tau)$ have distributions $g_{m\tau}(dM, k\tau)$.
Then, for any $K>0$ and a continuous bounded function $F(M)$,
\begin{equation}
 \label{eq1lemsubsdis}
\E [F(M(\Si))\1(\Si\in [1/K,K])]= \sum_{k \tau \in [1/K,K]} \int \, g_{k\tau}(dM, k\tau) F(M).
\end{equation}

Let $M_{\mu}(m\tau)$ be a Markov chain (with values in some metric space)
with a transition operator
\[
U^{\tau}_M F(\mu)=\int F(M) Q_{\tau} (\mu,dM).
\]
Moreover, let the pair $(M_{\mu}(m\tau), S_{\mu,s}(m\tau))$ (with $S \in \R$) also form a Markov chain with the
transition operator $U^{\tau}$ of the form
\[
U^{\tau}F(\mu,v)=\int F(M,S) P_{\tau} (\mu,s; S, dM) dS
=\int F(M,S) Q_{\tau} (\mu,dM) G_{\tau} (\mu,s; S) dS,
\]
with some density  $G_{\tau} (\mu,s; S)$.

Then the disttribution of the random vector $(M_{\mu}(k\tau), S_{\mu,s}(k\tau), S_{\mu,s}((k-1)\tau))$ is
\begin{equation}
 \label{eqjointdis}
\int Q_{\tau}(R,dM) G_{\tau}(R,S;W) P_{(k-1)\tau}(\mu,s; S, dR)\, dS
\end{equation}
(integration carried out over the variable $R$),  because
\[
\E F (M_{\mu}(k\tau), S_{\mu,s}(k\tau), S_{\mu,s}((k-1)\tau))
\]
\[
=\int \E (F(M_{\mu}(k\tau), S_{\mu,s}(k\tau), S)|M_{\mu}((k-1)\tau)=R, S_{\mu,s}((k-1)\tau)=S)P_{(k-1)\tau}(\mu,s; S, dR) dS
\]
\[
=\int F(M, W, S)Q_{\tau}(R,dM) G_{\tau}(R,S;W) \, dW P_{(k-1)\tau}(\mu,s; S, dR) \, dS.
\]

Let the coordinate $S_{\mu,s}(k\tau)$ be strictly increasing, and let
\[
Z_t^{\tau}=\sup\{ m\tau : S_{\mu,s}(m\tau)\le t\}=\inf\{ m\tau: S_{\mu,s}(m\tau)> t\}, \quad t>s,
\]
be its (generalized) inverse (or hitting times) process.

Then
\[
\E [F(M_{\mu}(Z_t^{\tau}))\1(Z_t^{\tau}\in [1/K,K])]
=\E \sum_{k \tau \in  [1/K,K]} F(M_{\mu}(k\tau)) \1(Z^{\tau}_t=k\tau)
\]
\[
=\E \sum_{k \tau \in  [1/K,K]} F(M_{\mu}(k\tau) \1(S_{\mu,s}((k-1)\tau)<t\le S_{\mu,s}(k\tau))
\]
\begin{equation}
 \label{eq4lemsubsdis}
 = \int F(M) \sum_{k \tau \in  [1/K,K]} \int \1(S<t\le W)
  Q_{\tau}(R,dM) G_{\tau}(R,S;W) \, dW P_{(k-1)\tau}(\mu,s; S, dR) \, dS.
\end{equation}

This can be rewritten as
\begin{equation}
 \label{eq5lemsubsdis}
\E [F(M_{\mu}(Z_t^{\tau}))\1(Z_t^{\tau}\in [1/K,K])]
 = \int \sum_{k \tau \in  [1/K,K]} (U^{\tau}F_S)(R,S) P_{(k-1)\tau}(\mu,s; S, dR)\, dS,
\end{equation}
with $F_S(R,W)= F(R)  \1(S<t\le W)$.

{\it Step 2}. In our case $U^{\tau}$ is given by \eqref{transprobmfMCtot2} and therefore,
for $R=\tau \sum_j\de_{z_j}$,
\eqref{eq5lemsubsdis} turns to the equation
\[
\E [F(M_{\mu}(Z_t^{\tau}))\1(Z_t^{\tau}\in [1/K,K])]
 = \int_{\MC^+_{\delta}(\R^d)} \int_0^t \sum_{k \tau \in  [1/K,K]}
 \int_{\R^+} \int_{\R^d} Q_{\z}(dr)\sum_i p_{\mu}(z_i, dy)\frac{a(z_i)}{A(\z)}
\]
\begin{equation}
 \label{eq6lemsubsdis}
 \times
F\left(R-\tau\de_{z_i}+\tau\de_{z_i+\sqrt{\tau}y}\right) \1(S\le t\le S+\tau^{1/\al \tau A(\z)}r)
 P_{(k-1)\tau}(\mu,s; S, dR)\, dS.
\end{equation}

We need to show that these expectations converge towards \eqref{eq12th2}.
By the density argument it is sufficient to show the convergence for smooth $F$ only.
One of the key observations is that the limit will be the same
if we change $F\left(R-\tau\de_{z_i}+\tau\de_{z_i+hy}\right)$ to just $F(R)$,
as the difference will be small as compared to the latter limit.
Therefore, due to the relation
\[
\sum_i \int p_{\mu}(z_i, dy)\frac{a(z_i)}{A(\z)}=1,
\]
we need to show the convergence of the expressions
\begin{equation}
 \label{eq7lemsubsdis}
 \int_{\MC^+_{\delta}(\R^d)} \int_0^t \sum_{k \tau \in  [1/K,K]}
 \int_{\R^+} Q_{\z}(dr)
F(R) \1(S\le t\le S+\tau^{1/\al \tau A(\z)}r)
 P_{(k-1)\tau}(\mu,s; S, dR)\, dS
\end{equation}
 towards \eqref{eq12th2}.

{\it Step 3}.
The key point for the argument is that, by Theorem \ref{th1}, the sums
 \begin{equation}
 \label{eq8lemsubsdis}
\tau \int_{\MC^+_{\delta}(\R^d)} \int_0^t \sum_{k \tau \in  [1/K,K]}
\Om(R,S) P_{(k-1)\tau}(\mu,s; S, dR)\, dS,
 \end{equation}

approximate uniformly the Riemannian sums for the integral
 \begin{equation}
 \label{eq9lemsubsdis}
 \int_0^t dS \int_{[1/K,K]} du
\, \Om(M_{\mu}(u),S) G(u; \mu,s; S),
 \end{equation}
 and thus converge to this integral for bounded continuous functions $\Om$.

It follows that we can slightly reduce the domain of integration in \eqref{eq7lemsubsdis}.
Namely, the limit of these expressions
will be the same as for the expressions

\[
 \int_{\MC^+_{\delta}(\R^d)}  \int_{[0,t-B\tau^{1/\al \tau A(\z)}]}  \sum_{k \tau \in  [1/K,K]}
\]
\begin{equation}
	\label{eq10lemsubsdis}
 \int_{\R^+} Q_{\z}(dr)
F(R) \1(S\le t\le S+\tau^{1/\al \tau A(\z)}r)
 P_{(k-1)\tau}(\mu,s; S, dR)\, dS.
\end{equation}

In fact the difference of these expressions with  \eqref{eq7lemsubsdis} is bounded in magnitude by
\[
 \|F\| \int_{\MC^+_{\delta}(\R^d)}  \int_{[t-B\tau^{1/\al \tau A(\z)},t]}  \sum_{k \tau \in  [1/K,K]}
 P_{(k-1)\tau}(\mu,s; S, dR)\, dS,
 \]
 which, according to \eqref{eq9lemsubsdis}, is of order
 $\tau^{1/\al \tau A(\z)} \tau^{-1}$, which tends to zero, as $\tau \to 0$, because $\al \tau A(\z)<1$.

Now let us rewrite the expressions \eqref{eq10lemsubsdis}  as
\[
 \int_{\MC^+_{\delta}(\R^d)} \int_{[0,t-B\tau^{1/\al \tau A(\z)}]}  \sum_{k \tau \in  [1/K,K]}
 \]
 \begin{equation}
 \label{eq11lemsubsdis}
 \int_{\R^+} Q_{\z}(dr)
F(R) [\theta_{\ge t}( S+\tau^{1/\al \tau A(\z)}r)-\theta_{\ge t}( S)]
 P_{(k-1)\tau}(\mu,s; S, dR)\, dS.
\end{equation}

The next key observation is that
\[
\theta_{\ge t}( S+\tau^{1/\al \tau A(\z)}r)-\theta_{\ge t}( S)=0
\]
for $S< t-B\tau^{1/\al \tau A(\z)}$ and $r<B$. Therefore, by \eqref{eq0apprrateCTRW},
\[
\int_{\R^+} Q_{\z}(dr)
[\theta_{\ge t}( S+\tau^{1/\al \tau A(\z)}r)-\theta_{\ge t}( S)]=
(\al \tau A(\z))^{-1}\int_0^{\infty} \theta_{\ge t}(S+r) r^{-1-\al \tau A(\z)} \, dr
\]
\[
=(\al \tau A(\z))^{-1}\int_{t-S}^{\infty} r^{-1-\al \tau A(\z)} \, dr
=(t-S)^{-\al \tau A(\z)}
\]
for $S< t-B\tau^{1/\al \tau A(\z)}$. Therefore expression \eqref{eq11lemsubsdis} rewrites as
\begin{equation}
 \label{eq12lemsubsdis}
 \int_{\MC^+_{\delta}(\R^d)} \int_{[0,t-B\tau^{1/\al \tau A(\z)}]}  \sum_{k \tau \in  [1/K,K]}
F(R) (t-S)^{-\al \tau A(\z)}
 P_{(k-1)\tau}(\mu,s; S, dR)\, dS.
\end{equation}
Noticing that changing the integration over  $[0,t-B\tau^{1/\al \tau A(\z)}]$ back to $[0,t]$
does not spoil the limit and using \eqref{eq9lemsubsdis} shows that expressions \eqref{eq12lemsubsdis}
converge to \eqref{eq12th2}, as was claimed.

(iii) Standard probabilistic arguments (using Dynkin's martingale) show that
any solution to equation \eqref{eq4th2} has the probabilistic representation $F(\mu,s)=\E F(\tilde M_{\mu,s}(t))$.
On the other hand, formula \eqref{eq11th2} allows one to check that this function does solve
problem \eqref{eq4th2}. These arguments are almost identical to the those used in \cite{Ko22} in a slightly
different situation, and we omit them here.

\section{Appendix: convergence rates for the standard CTRW}
\label{append}

For convenient referencing, we recall here the basic scheme of
jump-type approximations to stable generators.

 \begin{prop}
 \label{apprrateCTRW}
 Let $p(y)$ be a probability density on $\R_+$ such that
 $p(y)=y^{-1-\al}$ for $y\ge B$ with some $\al \in (0,1)$ and $B>0$. Then

 (i) for any bounded measurable $f$ having support on $[Bh,\infty)$,
\begin{equation}
 \label{eq0apprrateCTRW}
h^{-\al} \int_0^{\infty} f(hy) p(y) dy=\int_0^{\infty} \frac{f(y) dy}{y^{1+\al}},
\end{equation}

(ii) for any continuous $f$ on $\R_+$ such that $f(0)=f(\infty)=0$ and $f$ is Lipschitz
at zero so that $|f(y)| \le L y$ for $y\in [0,Bh]$ and some constant $L$, it follows that
 \begin{equation}
 \label{eq1apprrateCTRW}
 \left|h^{-\al} \int_0^{\infty} f(hy) p(y) dy-\int_0^{\infty} \frac{f(y) dy}{y^{1+\al}}\right| \le C_B L h^{1-\al} ,
\end{equation}
with
\[
C_B =\frac{B^{1-\al}}{1-\al} +\int_0^B y p(y)dy.
\]
 \end{prop}

\begin{proof}
(i) It follows from the conditions on the support of the function $f$ that we can replace $p(y)$ by $y^{-1-\al}$ and then make the change of variable $hy = y'$, implying \eqref{eq0apprrateCTRW}.

(ii) Let $f$ have support on $[0,Bh]$. Then
\[
\left|h^{-\al} \int_0^{\infty} f(hy) p(dy)\right|=\left|h^{-\al} \int_0^B f(hy) p(y) dy\right|
\le h^{1-\al} L \int_0^B y p(y) dy,
\]
and
\[
\left|\int_0^{\infty} \frac{f(y) dy}{y^{1+\al}}\right|=\left|\int_0^{Bh} \frac{f(y) dy}{y^{1+\al}}\right|
\le L \int_0^{Bh} \frac{dy}{y^{\al}}=\frac{B^{1-\al}}{1-\al}Lh^{1-\al},
\]
implying \eqref{eq1apprrateCTRW}.
\end{proof}

In particular, setting $\tau=h^{\al}$, it follows that
 \begin{equation}
 \label{eq10apprrateCTRW}
 \left|\tau ^{-1} \int_0^{\infty} (f(x \pm \tau^{1/\al} y)-f(x)) p(y) dy
 -\int_0^{\infty} \frac{(f(x\pm y)-f(x)) dy}{y^{1+\al}}\right| \le C_B L \tau^{(1/\al)-1},
\end{equation}
where $L$ is the sup of the derivative of $f$ near $x$.

\medskip

{\bf Acknowledgements.}
The authors would like to thank the Isaac Newton Institute for Mathematical Sciences, Cambridge, for support and hospitality during the programme Fractional Differential Equations Jan-Apr 2022, where work on this paper was undertaken.
	The first author is grateful to the Simons foundation for the support of his residence at INI in Cambridge
during the programme Fractional Differential Equations, Jan-Apr 2022.
		
The work of V.N. Kolokoltsov was supported by the Russian Science Foundation (project No. 20-11-20119),
the work of M.S. Troeva was supported by the Ministry of Science and Higher Education of the Russian Federation (Grant No. FSRG-2020-0006).

\end{document}